\documentclass[11pt, a4paper]{article}
\usepackage[english]{babel}
\usepackage{amssymb,amsmath,amsthm,bm,graphicx,natbib}

\newtheorem{thm}{Theorem}
\newtheorem{lem}{Lemma}

\theoremstyle{definition}

\newtheorem{rmk}{Remark}

\begin{document}

\title{\Large\bfseries Adaptive Priors based on Splines with Random  Knots}
\author{\bfseries Eduard Belitser and Paulo Serra \\
{\rm Department of Mathematics, Eindhoven University of Technology}}
\date{\today}
\maketitle

\begin{abstract}
Splines are useful building blocks when constructing priors on nonparametric models indexed by functions.
Recently it has been established in the literature that hierarchical priors based on splines with a random number of equally spaced knots and random coefficients in the B-spline basis corresponding to those knots lead, under certain conditions, to adaptive posterior contraction rates, over certain smoothness functional classes.
In this paper we extend these results for when the location of the knots is also endowed with a prior.
This has already been a common practice in MCMC applications, where the resulting posterior is expected to be more "spatially adaptive", but a theoretical basis in terms of adaptive  contraction rates was missing.
Under some mild assumptions,  we establish a result that provides sufficient conditions for adaptive contraction rates in a range of models.
\end{abstract}

\medskip

%\textbf{Subject Classifications:}

\section[Introduction]{Introduction}
\label{sec:intro}

The Bayesian approach in statistics has become quite popular in recent years as an alternative to classical \emph{frequentist} methods.
The main appeal of the Bayesian methodology is its conceptual simplicity:
given a model for the observed data $X \sim P_f$, $f\in\mathcal{F}$, some space of functions, put a prior on the unknown parameter $f$ and draw inferences based on the resulting posterior $\Pi(f|X)$.
Knowledge about the model under study can also be be incorporated into the inference procedure via the prior.
However, some seemingly "correct" priors can lead to unreasonable posteriors, especially in nonparametric models. 
It is therefore desirable to place ourselves in a setting where it is possible to assess the quality of the resulting posterior from some objective point of view.

This gave rise to the development of the notion of contraction rate (cf. \citet{Ghosal:2000}), a Bayesian analog of a convergence rate:
data is assumed to come from a fixed probability measure $P_0 = P_{f_0}$ for a "true" $f_0 \in \mathcal{F}$;
the contraction rate is then the smallest radius such that the posterior mass in a Hellinger ball of probability measures around $P_0$ converges to 1 in $P_0$-probability as some information index such as a sample size goes to infinity.

Some general results about posterior contraction rates establish sufficient conditions on prior distributions such that the resulting posteriors attain a certain contraction rate.
In this spirit, when studying specific priors, some authors now choose to present their results in the form of say {\it meta-theorems} which claim that sufficient conditions (such as the ones in \citet{Ghosal:2000}) required to attain a certain range of contraction rates hold for their choice of prior;
cf.\ \citet{deJonge:2012}, \citet{Shen:2012}, \citet{Vaart:2008b} and further references therein.
We adopt this practice here as well.
 
In the case where $f_0$ is a function from some functional space of smoothness $\alpha$, the posterior contraction rate is typically compared to the convergence rate of the minimax risk (called optimal rate) over that space in the estimation problem.
For example, if we observe a sample of size $n$ and want to estimate a univariate $\alpha$-smooth function (e.g., density or regression function), the typical optimal rate is of order $n^{-\alpha/(2\alpha+1)}$, possibly up to a logarithmic factor depending on the risk function.
If the smoothness parameter $\alpha$ is unknown, and one wants to build estimators which attain the optimal rate corresponding to $\alpha$ but do not depend explicitly on $\alpha$, one speaks of an adaptation problem.
In a Bayesian context, the adaptation problem consists in finding a prior which leads to the optimal posterior contraction rate (usually up to a logarithmic factor) for any $\alpha$-smooth function of interest and does not depend on the smoothness parameter $\alpha$.
Such priors are called rate adaptive.
There is a growing number of papers, where this problem has been studied in different settings;
cf.\ \citet{deJonge:2012}, \citet{Shen:2012}, \citet{Vaart:2008b}, \citet{Vaart:2009} and \citet{Belitser:2003} among others. 

Splines, in particular, can be used when constructing adaptive priors.
A spline (cf.\ \citet{deBoor:1978}) is a piecewise polynomial function designed to have a certain level of smoothness which is referred to as its order. 
Splines are easy to store, differentiate, integrate and evaluate on a computer, and are extensively used in practice for constructing good, parsimonious  approximations of smooth functions.
The points at which the different polynomial pieces of a spline connect are called knots. 
If an order (read: maximal polynomial degree) and a set of knots is fixed, then the space of all splines with that order and those knots forms a linear space which admits a basis of so called B-splines.
Any spline of a fixed order is consequently characterized by a set of knots and its coordinates in the B-splines basis corresponding to those knots.
Randomly generating a number of knots and, given those, generating random coordinates in the corresponding B-spline basis with equally spaced knots results in a random spline whose law can be used as a prior.
If, given the number of knots, the coordinates in the corresponding B-spline basis are chosen to be independent and normally distributed, then the resulting spline has a conditionally Gaussian law and was studied by \citet{deJonge:2012} by using Reproducing Kernel Hilbert Space techniques. 
\citet{Shen:2012} propose a more general, random series prior:
the coefficients in the series are not necessarily independent or Gaussian and a basis other than the B-spline basis can also be used.

The case where the locations of the knots are also random is not covered by the results of either \citet{deJonge:2012} or \citet{Shen:2012}.
However when practitioners put a prior on the number of knots they almost invariably also put a prior on the locations of the knots (e.g., \citet{Denison:1998}, \citet{diMatteo:2001}, \citet{Sharef:2010}) -- a Poisson process is a popular choice.
Their motivation for allowing arbitrarily located knots seems to be twofold.
Firstly, this is attractive from the implementation point of view:
designing reversible jump MCMC samplers is much simpler if any collection of knots is allowed since new knots can be inserted at arbitrary positions causing only localized changes in the spline.
Secondly, the resulting posterior based on the prior with random locations of the knots is expected to be more "spatially adaptive":
the function of interest may not have a fixed level of smoothness throughout its support, it may consist of rough and smooth pieces.
To sustain an adequate level of accuracy over the whole support, more knots are needed in rough pieces and less in smooth ones.
Therefore, to make it at least possible for the resulting posterior to pick up eventual spatial features of the function, the prior has to be flexible enough to model random locations of the knots.

In this paper, we extend the results of \citet{deJonge:2012}, and those of \citet{Shen:2012} in respect to the prior with random knots:
we add one more level to the hierarchical spline prior by putting a prior on the location of the knots of the spline as well, making, in fact, the basis functions also random.
Under some mild assumptions on the proposed hierarchical spline prior,
we establish our main result for the proposed prior, providing sufficient conditions for adaptive, optimal contraction rates of the resulting posterior in a range of models (among others: density estimation, nonparametric regression, binary regression, Poisson regression, and classification).
In doing so, we provide a theoretical basis for the common practice of using randomly located knots in spline based priors.

\section[Preliminaries]{Notation and preliminaries on splines}
\label{sec:notation}

First we introduce some notation.
For $d\in\mathbb{N}$ and $1\le p < \infty$ denote by $\|\bm{x}\|_p = \big(\sum_{i=1}^d |x_i|^p\big)^{1/p}$ the $l_p$-norm of $\bm{x}=(x_1,\ldots, x_d) \in\mathbb{R}^d$ and by $\|\bm{x}\|_\infty = \max_{i=1,\dots,d} |x_i|$.
For $1\le p < \infty$ let the $L_p$-norm of a function $f$ on [0,1] be $\|f\|_p = \big(\int_0^1 |f(x)|^p\, dx)^{1/p}$ and $\|f\|_\infty = \sup_{x\in[0,1]}|f(x)|$.

We use $\lesssim$ (respectively $\gtrsim$) to denote smaller (respectively greater) or equal up to a constant,  the symbols $a \vee b$ and $a \wedge b$ stand for $\max\{a,b\}$ and $\min\{a,b\}$ respectively.
The covering number $N(\epsilon, S, d)$ of a subset $S$ of a metric space with balls of size $\epsilon$  is the smallest number of balls (with respect to distance $d$) of radius $\epsilon$ needed to cover $S$.

Now we provide some preliminaries on splines, which can be found, for example, in \citet{Schumaker:2007}. 
A function is called a spline is of order $q\in\mathbb{N}$, with respect to a certain partition of its support, if it is $q-2$ times continuously differentiable and when restricted to each interval in this partition, coincides with a polynomial of degree at most $q-1$.
Consider $q\in\mathbb{N}$, $q\ge2$, which will be fixed throughout the remainder of this text.
For any $j\in\mathbb{N}$, such that $j\ge q$ let $\mathcal{K}_j = \{( k_1, \dots, k_{j-q} )\in (0, 1)^{j-q}: 0<k_1< \dots < k_{j-q}<1 \}$.
We will refer to a vector $\bm{k}=\bm{k}_j\in\mathcal{K}_j$ as a set of inner knots; the index $j$ in $\bm{k}_j$ will sometimes be used to emphasize the dependence on $j$.
A vector $\bm{k}\in\mathcal{K}_j$ will be said to induce the partition $\big\{ [k_0, k_1), [k_1, k_2), \dots, [k_{j-q}, k_{j-q+1}]\big\}$, with $k_0 = 0$ and $k_{j-q+1} = 1$. 
For any $\bm{k} \in\mathcal{K}_j$ we will call $M(\bm{k}) = \max_{i=1}^{j-q+1} |k_i-k_{i-1}|$ the mesh size of the partition induced by $\bm{k}$ and $m(\bm{k}) = \min_{i=1}^{j-q+1} |k_i-k_{i-1}|$ the sparseness of the partition induced by $\bm{k}$.
For a $\bm{k} \in\mathcal{K}_j$, denote by $\mathcal{S}^{\bm{k}}=\mathcal{S}^{\bm{k}}_q$ the linear space of splines  of order $q$ on $[0, 1]$ with simple knots $\bm{k}$ (see the definition of knot multiplicity in \citet{Schumaker:2007}).
This space has dimension $j$ and admits a basis of so called B-splines $\{B_1^{\bm{k}},\ldots, B_j^{\bm{k}}\}$. 
The construction of $\{B_1^{\bm{k}},\ldots, B_j^{\bm{k}}\}$  involves the knots $k_{-q+1},\ldots, k_{-1}, k_0, k_1, \ldots, k_{j-q}, k_{j-q+1}, k_{j-q+2},\ldots,k_j$, with arbitrary extra knots $k_{-q+1} \le \cdots \le k_{-1} \le k_0 =0$ and $1=k_{j-q+1} \le k_{j-q+2} \le \cdots \le k_j$. 
Usually one takes $k_{-q+1}=\cdots =k_{-1}= k_0 =0$ and $1=k_{j-q+1}=\cdots =k_j$, and we adopt this choice here as well.
These basis functions are nonnegative: $B_i^{\bm{k}}(x) \ge 0$, for all $x\in[0,1]$.
Besides, they have local support and form a partition of unity:
\begin{equation}
\label{local_support}
B_i^{\bm{k}}(x)= 0 \;\; \mbox{for } x\not\in [k_{-q+i},k_i],	\quad	\sum_{i=1}^{j} B_i^{\bm{k}}(x)=1 \; \; \mbox{for all } x\in[0,1].
\end{equation}
To refer explicitly to the coordinates $\bm{a}=(a_1,\ldots, a_j) \in\mathbb{R}^{j}$ of a spline on a specific B-spline basis with inner knots $\bm{k}$, we write $s_{\bm{a},\bm{k}}(x) = \sum_{i=1}^j a_i B_i^{\bm{k}}(x)$, $x \in [0,1]$.
Since $\sum_{i=1}^{j} B_i^{\bm{k}}(x)=1$, it is easy to see that for any $s_{\bm{a},\bm{k}}, s_{\bm{b},\bm{k}}\in \mathcal{S}^{\bm{k}}_q$
\begin{equation}
\label{N1}
\|s_{\bm{a},\bm{k}} - s_{\bm{b},\bm{k}}\|_2 \le
\|s_{\bm{a},\bm{k}} - s_{\bm{b},\bm{k}}\|_\infty
\le\|\bm{a} - \bm{b}\|_\infty
\le \|\bm{a} - \bm{b}\|_2.
\end{equation}

Splines have good approximation properties for sufficiently smooth functions provided they are defined on a partition with appropriately small mesh size.
We say that a function $f$ on $[0,1]$  belongs to a generic smoothness class $\mathcal{F}_\alpha$, $\alpha>0$,  if for any set of inner knots $\bm{k}$ there exists a spline $s_{\bm{a},\bm{k}} \in \mathcal{S}^{\bm{k}}_q$ such that for some bounded $C_f$
\begin{equation}
\label{Ap}
\|f-s_{\bm{a},\bm{k}}\|_\infty \le C_f M^{\alpha}(\bm{k}).
\end{equation}
We will also be assuming that $\mathcal{F}_\alpha$ is contained in a Lipschitz class:
$\mathcal{F}_\alpha \subseteq \mathcal{L}(\kappa_\alpha, L_\alpha) = \{f: |f(x_1)-f(x_2)| \le L_\alpha |x_1-x_2|^{\kappa_\alpha}, x_1,x_2 \in [0,1] \}$ for some $\kappa_\alpha, L_\alpha>0$.

A leading example of a smoothness class $\mathcal{F}_\alpha$ is the H\"older space  $\mathcal{H}_\alpha=\mathcal{H}_\alpha(L, [0,1])$, $0<\alpha\le q$, which is the collection of all functions $f$ that have bounded derivatives up to order $\alpha_0=\lfloor\alpha\rfloor = \max\{z\in \mathbb{Z}: \, z <\alpha\}$ and such that the $\alpha_0$-th derivative satisfies the H\"older condition $|f^{(\alpha_0)}(x) - f^{(\alpha_0)}(y)| \le L |x-y|^{\alpha-\alpha_0}$, for $L>0$ and $x,y\in[0,1]$.
In this case, a well-known spline approximation result (cf.\ \citet{deBoor:1978}) claims that (\ref{Ap}) holds with $C_f= C_q \|f^{(\alpha)}\|_\infty$ for some constant $C_q$ depending only on $q$.
Other examples of smoothness classes for which the approximation property (\ref{Ap}) hold, include $\alpha$-times continuously differentiable functions, Sobolev and Besov spaces; cf.\ Theorems 6.21, 6.25 and 6.31 in \citet{Schumaker:2007}.

\section[Main Result]{Main Result}
\label{sec:result}

We begin by describing a hierarchical prior on $\mathcal{S}=\mathcal{S}_q=\cup_{j=q}^\infty \cup_{\bm{k}\in\mathcal{K}_j}\mathcal{S}^{\bm{k}}_q$:
first draw a number $J\in \mathbb{N}$, $J\ge q$; then, given $J$, generate independently $(J-q)$ inner knots $\bm{K}_J\in\mathcal{K}_j$ and also independently, $J$ B-spline coefficients $\bm{\theta}\in\mathbb{R}^J$.
Our prior on $\mathcal{S}$ will be the law of the random spline $s_{\bm{\theta},\bm{K}_J}$.
We impose the following conditions on this prior.
For $c_1, c_2>0$, $0 \le t_1, t_2 \le 1$ and all sufficiently large $j$,
\begin{align}
\label{A1-1}
\mathbb{P}( J > j)	&	\lesssim	\exp\big( - c_1 j \log^{t_1} j \big),\\
\label{A1-2}
\mathbb{P}(J =j)	&\gtrsim		\exp\big( - c_2 j \log^{t_2} j \big).
\end{align}
For some $\tau\ge 1$, $c_3>0$, $0 \le t_3 \le 1$, and all $j\ge q$,
\begin{align}
\label{A0}
\mathbb{P}\big( m(\bm{K}_j) < \delta(j) | J = j \big)	&	=		0,\\
\label{A2}
\mathbb{P}\big( M(\bm{K}_j) \le \tau/j|  J = j \big)	&	\gtrsim	\exp\big(- c_3 j \log^{t_3}j\big),
\end{align}
where $\delta(i)$ is a positive, strictly decreasing function on $\mathbb{N}$.
Without loss of generality assume that $\delta(i) \le 1$, $i \in \mathbb{N}$.
For each $j\ge q$, the conditional distribution of $\bm{\theta}\in\mathbb{R}^j$ satisfies the following condition: for any $M>0$ there exists $c_0=c_0(M)$ such that 
\begin{equation}
\label{A3}
\mathbb{P}\big(\|\bm{\theta} - \bm{\theta}_0\|_\infty \le \epsilon |  J = j \big)	\gtrsim	\exp\big( - c_0 j \log(1/\epsilon) \big)
\end{equation}
for all $\epsilon>0$ and all $\bm{\theta}_0\in\mathbb{R}^j$ such that $\|\bm{\theta}_0\|_\infty \le M$.

For examples of particular choices on the components of our hierarchical prior which verify these conditions we refer the reader to Section \ref{examples}.

Denote $\mathcal{C}^j(M) =[-M,M]^j$.
The following theorem is our main result.

\begin{thm}
\label{theo}
Let $\|f_0\|_\infty < M$ and $f_0 \in \mathcal{F}_{\alpha}$ so that (\ref{Ap}) holds with $C_{f_0}$.
Let $\epsilon_n,\bar{\epsilon}_n$ be two positive sequences such that $\epsilon_n \ge \bar{\epsilon}_n$, $\epsilon_n \rightarrow 0$ as $n\rightarrow \infty$ and $n\bar{\epsilon}_n^2>1$.
Assume that there exist sequences $J_n, \bar{J}_n>q$, $M_n>0$ and a constant $c_M \ge c_1$ satisfying:
\begin{align}
\label{C1}
J_n \log\Big[\frac{J_n (M_n \vee 1) }{\epsilon_n \delta(J_n)}\Big]	&	\lesssim 	n\epsilon_n^2,\\
\label{C2}
\frac{n\bar{\epsilon}_n^2}{\log^{t_1}J_n}\le J_n, \;\;	P\big(\bm{\theta}\not\in \mathcal{C}^j(M_n)| J = j \big)	&\lesssim
\exp(- c_M n\bar{\epsilon}_n^2),\; q \le j \le J_n,\\
\label{C4}
\Big[\frac{\bar{\epsilon}_n}{\tau^\alpha C_{f_0}}\Big]^{-1/\alpha}\le \bar{J}_n, \quad \log^{t_2 \vee t_3} \bar{J}_n &\lesssim\log\frac1{\bar{\epsilon}_n}. 
\end{align}
Let  $\mathcal{S}_n=\cup_{j=q}^{J_n} \cup_{\bm{k}\in \mathcal{K}^{\delta(j)}_j} \big\{ s_{\bm{\theta}, \bm{k}}\in \mathcal{S}_q^{\bm{k}}: \|\bm{\theta}\|_\infty \le M_n \big\}$, where $\mathcal{K}_j^\delta=\{ \bm{k}\in\mathcal{K}_j: m(\bm{k}) > \delta\}$.
Then it holds that
\begin{align}
\label{S1}
\log N(\epsilon_n, \mathcal{S}_n,\|\cdot\|_2)							&	\lesssim	n\epsilon_n^2,\\
\label{S2}
P\big(s_{\bm{\theta}, \bm{K}_J}  \not\in \mathcal{S}_n\big)				&	\lesssim	\exp\big(- c_1 n\bar{\epsilon}_n^2 \big),\\
\label{S3}
P\big(\|s_{\bm{\theta}, \bm{K}_J} - f_0\|_\infty \le 2 \bar{\epsilon}_n \big)	&	\gtrsim	\exp\big\{-(c_0(M)+c_2+c_3) \bar{J}_n\log(1/\bar{\epsilon}_n)\big\}.
\end{align}
\end{thm}

\begin{rmk}
Consider constants $c_4, c_5>0$ and a function $\delta(\cdot)$ as above.
If condition (\ref{A0}) is replaced by
\begin{align}
\label{A0'}
\sum_{j=q}^{J_n} \mathbb{P}\big( J = j \big) \mathbb{P}\big( m(\bm{K}_j) < \delta(j) | J = j \big)	&	\le		c_5 \exp(-c_4 n),\tag{6'}
\end{align}
then the conclusions of Theorem \ref{theo} remain valid so long as $J_n$ is a sequence satisfying (\ref{C1}) and (\ref{C2})
(cf.\ Section \ref{examples} and Remark \ref{rem} for a comparison of (\ref{A0}) and (\ref{A0'}).)
\end{rmk}

\begin{proof}
First we establish (\ref{S1}).
Let $L_n(j)=4M_nj(q+1)(\delta(j))^{-(q+1)}$ and $j>q$.
Let $\{\bm{\theta}_1, \ldots, \bm{\theta}_{m_1}\}$ be an $\epsilon_n/2$-net of the set $\{\bm{\theta} \in\mathbb{R}^j: \|\bm{\theta}\|_\infty\le M_n\}$ and let $\{\bm{x}_1, \ldots, \bm{x}_{m_2}\}$ be an $\epsilon_n/(2L_n(j))$-net of $\{\bm{x}\in\mathbb{R}^{j-q}:\bm{x} \in (0,1)^{j-q}\}$, both with respect to the $\|\cdot\|_\infty$-norm.
Then, by using (\ref{N1}) and Lemma \ref{lemma1} (Lemma \ref{lemma1} is applicable since $\epsilon_n/(2L_n(j))\le \delta(j)$ for sufficiently large $n$), $\{s_{\bm{\theta_k},\bm{x}_l}, k=1,\ldots,m_1, \; l=1,\ldots m_2\}$ forms an $\epsilon_n$-net of $\cup_{\bm{k}\in \mathcal{K}^{\delta(j)}_j}\big\{ s_{\bm{\theta}, \bm{k}}\in \mathcal{S}_q^{\bm{k}}:\, \|\bm{\theta}\|_\infty \le M_n \big\}$ with respect to the $\|\cdot\|_\infty$-norm.
By using this fact, we obtain
\begin{align*}
& 	N\big(\epsilon_n, \mathcal{S}_n, \|\cdot\|_2\big)	\le	N\big(\epsilon_n,\mathcal{S}_n, \|\cdot\|_\infty \big) \\
& \le	\sum_{j=q}^{J_n} N\Big(\epsilon_n, \cup_{\bm{k}\in \mathcal{K}^{\delta(j)}_j} \big\{s_{\bm{\theta}, \bm{k}}\in \mathcal{S}_q^{\bm{k}}:\, \|\bm{\theta}\|_\infty \le M_n \big\},\|\cdot\|_\infty\Big)\\
&\le	\sum_{j=q}^{J_n} \Big[N\Big(\frac{\epsilon_n}{2}, \big\{\bm{\theta} \in \mathbb{R}^j: \|\bm{\theta}\|_\infty \le M_n \big\},\|\cdot\|_\infty\Big) N\Big(\frac{\epsilon_n}{2L_n(j)}, (0,1)^{j-q}, \|\cdot\| \infty\Big) \Big]\\
& \le J_n \Big[\frac{2(M_n\vee 1)}{\epsilon_n}\Big]^{J_n} \Big[\frac{2L_n(J_n)}{\epsilon_n} \Big]^{J_n -q}\le J_n \Big(\frac{16(q+1) (M_n \vee 1)^2 J_n}{\epsilon_n^2 (\delta(j))^{q+1}} \Big)^{J_n}.
\end{align*}
The last relation and (\ref{C1}) imply (\ref{S1}):
\[
\log N \big(\epsilon_n, \mathcal{S}_n, \|\cdot\|_2\big)	\lesssim	J_n \log \Big[\frac{J_n(M_n \vee 1)}{\epsilon_n \delta(J_n)} \Big] \lesssim n\epsilon_n^2.
\]

Now we check (\ref{S2}).
From the definition of $\mathcal{S}_n$, the relations (\ref{A1-1}), (\ref{A0}) and (\ref{C2}), it follows that
\begin{align*}
\mathbb{P}\big(s_{\bm{\theta}, \bm{K}_J} \not\in\mathcal{S}_n\big)	&\le	\mathbb{P}\big( J > J_n \big) + \sum_{j=q}^{J_n} \mathbb{P}\big( J = j \big) \mathbb{P}\big( m(\bm{K}_j) < \delta(j) | J = j \big)\\
&+ \sum_{j=q}^{J_n} \mathbb{P}\big( J = j \big)\mathbb{P}\big( \bm{\theta} \not\in\mathcal{C}^j(M_n)|J = j \big)\\
&\lesssim \exp\big\{-c_1 J_n \log^{t_1} J_n \big\} + 0 + \exp\big\{-c_M n\bar{\epsilon}_n^2 \big\}\\
&\lesssim \exp\big\{-c_1 n\bar{\epsilon}_n^2 \big\}.
\end{align*}

It remains to prove (\ref{S3}).
First note that, by using (\ref{Ap}) and (\ref{C4}), for all $j\ge \bar{J}_n$ and for all sets of knots $\bm{k}_j\in\mathcal{K}_j$ such that $M(\bm{k}_j) \le \tau/ j$, there exists a spline $s_{\bm{\theta}_0,\bm{k}_j}\in \mathcal{S}_q^{\bm{k_j}}$ (of course, $\bm{\theta}_0=\bm{\theta}_0(\bm{k}_j)=\bm{\theta}_0(\bm{k}_j, f_0))$ such that  
\begin{equation}
\label{theorem1_a}
\|f_0 - s_{\bm{\theta}_0,\bm{k}_j}\|_\infty \le C_{f_0}  M^\alpha(\bm{k}_j) \le C_{f_0} \tau^\alpha  \bar{J}_n^{-\alpha} \le \bar{\epsilon}_n.
\end{equation}

Since $\|f_0\|_\infty < M$ and $\bar{J}_n$ must grow with $n$ in view of (\ref{C4}), it follows from Lemma \ref{lemma2} and (\ref{theorem1_a}) $\|\bm{\theta}_0(\bm{k}_j)\|_\infty \le M$ for all $\bm{k}_j\in\mathcal{K}_j$ such that $M(\bm{k}_j) \le \tau /\bar{J}_n$ for $j\ge \bar{J}_n$.

Introduce the events:
$E_1^j=\{M(\bm{K}_j) \le \tau/j\}$,
$E_2^j=\{\|f_0 - s_{\bm{\theta}_0(\bm{K}_j),\bm{K}_j}\|_\infty \le \bar{\epsilon}_n\}$, 
$E_3^j=\{\|\bm{\theta}_0(\bm{K}_j)- \bm{\theta}\|_\infty \le \bar{\epsilon}_n\}$,
$E_4^j=\{\|f_0-s_{\bm{\theta}, \bm{K}_j}\|_\infty \le 2\bar{\epsilon}_n\}$ and
$E_5^j=\{\|\bm{\theta}_0(\bm{K}_j)\|_\infty \le M\}$.
Using the argument from the previous paragraph, the triangle inequality, (\ref{N1}) and (\ref{theorem1_a}),  we obtain that
\begin{equation}
\label{theorem1_b}
E^{\bar{J}_n}_1 \subseteq E^{\bar{J}_n}_2, \quad	E^{\bar{J}_n}_1 \subseteq E^{\bar{J}_n}_5, \quad	E^j_2 \cap E^j_3 \subseteq E^j_4, \quad j \ge q.
\end{equation}
Combining (\ref{A1-2}), (\ref{A2}), (\ref{A3}), (\ref{C4}) and (\ref{theorem1_b}), we prove (\ref{S3}):
\begin{align*}
\mathbb{P}& \big(\|s_{\bm{\theta}, \bm{K}_J} - f_0\|_\infty \le 2 \bar{\epsilon}_n \big) = \mathbb{P}(E^{J}_4) \ge \mathbb{P}(J=\bar{J}_n) \mathbb{P}\big(E^{\bar{J}_n}_4 |J=\bar{J}_n)\\
&\ge	\mathbb{P}(J=\bar{J}_n) \mathbb{P}\big(E^{\bar{J}_n}_2 \cap E_3^{\bar{J}_n} |J=\bar{J}_n)\\
&\ge	\mathbb{P}(J=\bar{J}_n) \mathbb{P}\big(E^{\bar{J}_n}_1 \cap E^{\bar{J}_n}_3 \cap E^{\bar{J}_n}_5 |J = \bar{J}_n \big)  \\
&=	\mathbb{P}(J=\bar{J}_n) \mathbb{E}\big[P\big(E^{\bar{J}_n}_1 \cap E^{\bar{J}_n}_3 \cap E^{\bar{J}_n}_5|J = \bar{J}_n,\bm{K}_{\bar{J}_n} \big) \big]  \\
&=	\mathbb{P}(J=\bar{J}_n) \mathbb{E}\big[\mathbb{I}\{\bm{K}_{\bar{J}_n}\in E^{\bar{J}_n}_1\cap E^{\bar{J}_n}_5 \}\mathbb{P}\big(E^{\bar{J}_n}_3|J = \bar{J}_n,\bm{K}_{\bar{J}_n} \big) \big]  \\
&\ge	\mathbb{P}(J=\bar{J}_n) \mathbb{P}\big(E^{\bar{J}_n}_1 | J = \bar{J}_n\big)\inf_{ \|\bm{\theta}_0\|_\infty \le M} \mathbb{P}\big(\|\bm{\theta} - \bm{\theta}_0\|_\infty \le \bar{\epsilon}_n | J = \bar{J}_n \big) \\
& \gtrsim\exp\big( - (c_2+c_3) \bar{J}_n \log^{t_2\vee t_3} \bar{J}_n \big) \exp\big( - c_0(M) \bar{J}_n \log (1/{\bar{\epsilon}_n}) \big)\\
& \gtrsim \exp\big(- (c_0(M)+c_2+c_3)\bar{J}_n\log(1/\bar{\epsilon}_n)\big).
\end{align*}
\end{proof}

\begin{rmk}
If the range of the underlying curve $f_0$ is contained in some known interval $[a,b]\subset \mathbb{R}$, then, according to Lemma \ref{lemma2} and the proof of property (\ref{S3}), the prior on
$\bm{\theta}\in \mathbb{R}^j$ can be chosen to be supported on, say, $[a-1,b+1]^j$ so that (\ref{A3}) has to hold only for $\bm{\theta}_0\in [a-1,b+1]^j$.
Condition (\ref{C2}) will trivially be satisfied for $M_n> (1-a)\wedge(b+1)$.
\end{rmk}

\begin{rmk}
If (\ref{A4}) is assumed instead of (\ref{A2}), the proof of (\ref{S3}) can then be simplified a lot, as in this case one can condition on the event $\{\bm{K}_{\bar{J}_n}=\bar{\bm{k}}_{\bar{J}_n}\}$ so that $\bm{\theta}_0=\bm{\theta}_0(\bar{\bm{k}}_{\bar{J}_n})$ becomes fixed and $\mathbb{P}(E_1^J|J=\bar{J}_n,\bm{K}_{\bar{J}_n}=\bar{\bm{k}}_{\bar{J}_n}) =1$.
\end{rmk}

\begin{rmk}
\label{rem}
Condition (\ref{A0}) is used in the proof of Theorem \ref{theo} exclusively to enforce $\sum_{j=q}^{J} \mathbb{P}\big( J = j \big) \mathbb{P}\big( m(\bm{K}_j) < \delta(j) | J = j \big)$ to be zero.
Inspection of the proof shows, however, that it would suffice to require this sum to be upper-bounded by a multiple of $\exp\big\{-c_1 n\bar{\epsilon}_n^2 \big\}$.
Although this would be a weaker requirement, typically the sequence $\bar{\epsilon}_n$ will depend on the unknown smoothness $\alpha$.
Note however that since $\epsilon_n\ge\bar{\epsilon}_n$ and $\epsilon_n$ will obviously be taken to converge to 0, then for large enough $n$, $c_1 n\bar{\epsilon}_n^2 < n$.
This allows the term $\sum_{j=q}^{J} \mathbb{P}\big( J = j \big) \mathbb{P}\big( m(\bm{K}_j) < \delta(j) | J = j \big)$ to be absorbed into the remaining terms of the bound on $\mathbb{P}\big(s_{\bm{\theta}, \bm{K}_J} \not\in\mathcal{S}_n\big)$ in the proof.
Consequently, as claimed, Theorem \ref{theo} also holds if (\ref{A0'}) is assumed instead of (\ref{A0}).
\end{rmk}

\section[Implications]{Implications of the main result}
\label{application}

We clarify now the relevance of our result.
Consider a family of models $\mathcal{P} = \big\{P_f : f\in\mathcal{F}_{\mathcal{A}}\big\}$, $\mathcal{F}_{\mathcal{A}}=\cup _{\alpha \in \mathcal{A}}\mathcal{F}_{\alpha}$, with densities $p_f$ with respect to some common dominating measure.
Assume that we observe a sample $\bm{X}^{(n)}=(X_1,\ldots, X_n)\sim p_{f_0}^{(n)}$, $X_i \stackrel{\tiny ind}{\sim} p_{f_0}$, $f_0 \in \mathcal{F}_\alpha$ for some unknown smoothness $\alpha \in \mathcal{A}$.
The Bayesian approach consists of putting a prior measure $\Pi$ on $\mathcal{F}\subseteq\mathcal{F}_{\mathcal{A}}$ which, together with the likelihood $p_f^{(n)}$, 
leads to the posterior distribution $\Pi(\cdot| \bm{X}^{(n)})$ via Bayes' formula:
\[
\Pi\big(A|\bm{X}^{(n)}\big)= \frac{\int_A p_f^{(n)}(\bm{X}^{(n)})\, d\Pi(f)}{\int_{\mathcal{F}} p_f^{(n)}(\bm{X}^{(n)})\, d\Pi(f) }
\]
for a measurable $A\subseteq\mathcal{F}$.
The asymptotic behavior of the posterior distribution can be studied from the point of view of the probability measure $P_0=P_{f_0}$; see \citet{Ghosal:2000}.

For two densities $p_f$ and $p_g$ with $f, g \in \mathcal{F}_{\mathcal{A}}$, define the (squared) Hellinger metric $h^2(p_f, p_g)$ =
$2\big(1-\mathbb{E}_{g}\sqrt{p_f(X)/p_g(X)}\big)$, Kullback-Leibler divergence $K(p_f,p_g)$ = $-\mathbb{E}_g\log\big(p_f(X)/p_g(X)\big)$ and the Csisz\'ar f- divergence $V(p_f, p_g)$ = $\mathbb{E}_g\log^2\big(p_f(X)/p_g(X) \big)$.
Define also the ball $B(\epsilon_n,f_0)=\big\{f\in\mathcal{F}: K(f,f_0) \le \epsilon^2,V(f,f_0) \le \epsilon^2 \big\}$.

The following theorem is the main result  of \citet{Ghosal:2000} (for a version involving two sequences $\epsilon_n$ and $\bar\epsilon_n$ cf.\ also \citet{Ghosal:2001}) which makes a statement about the asymptotic behavior of a posterior measure.

\begin{thm}[Theorem 2.1 of \citet{Ghosal:2000}]
\label{theorem:contraction}
Suppose that for two positive sequences $\epsilon_n\ge \bar\epsilon_n$ such that $n\bar{\epsilon}_n^2>1$ and $\epsilon_n \rightarrow 0$ as $n\rightarrow \infty$, sets $\mathcal{F}_n\subseteq\mathcal{F}$ and constants $c_1,c_2,c_3,c_4>0$, the following conditions hold:
\begin{align}
\label{S4}
\log N \big(\epsilon_n,\mathcal{F}_n,h\big)	&	\le	c_1 n\epsilon_n^2,\\
\label{S5}
\Pi(\mathcal{F}\backslash\mathcal{F}_n)		&	\le	c_2 e^{-(c_3+4)n\bar{\epsilon}_n^2},\\
\label{S6}
\Pi(B(\bar{\epsilon}_n,f_0) )				&	\ge	c_4 e^{-c_3n\bar{\epsilon}_n^2}.
\end{align}
Then, for large enough $M>0$, $\Pi\big(f\in \mathcal{F}: h(p_f, p_{f_0})\ge M\epsilon_n | \bm{X}^{(n)}\big)\rightarrow 0$  as $n\to \infty$ in $P_{f_0}$-probability.
\end{thm}

The conditions of this theorem require the existence of a \emph{sieve} $\mathcal{F}_n$ with small entropy (\ref{S4}) which contains most of the prior mass (\ref{S5}) and which enough prior mass around the parameter $f_0$ which indexes the "true" underlying measure of the data.
Assume now that the models in $\mathcal{P}$ are such that for $d^2$ being $h^2$, $K$ or $V$, $d^2(p_f,p_{f_0})\lesssim \|f-f_0\|_2^2$.
If in addition one can prove that in the considered model $h(p_f,p_{f_0}) \gtrsim\|f-f_0\|_2$, then Theorem \ref{theorem:contraction} delivers a contraction rate $\epsilon_n$ with respect to the $L_2$-distance as well.
Some examples of models for which the above relations between norms can be established are, among others, density estimation, non-parametric regression, binary regression, Poisson regression and classification; cf.\ \citet{Ghosal:2000}, \citet{deJonge:2012}, \citet{Shen:2012}.
In this case one can apply our meta-theorem (Theorem \ref{theo}) to obtain an adaptive contraction rate which essentially verifies (\ref{S4})--(\ref{S6}) for our spline-based prior.
We summarize this in the following theorem.

\begin{thm}
\label{theo:main}
Let $\Pi$ be the spline prior described in Section \ref{sec:result}.
Consider a family of models $\mathcal{P} = \big\{P_f : f\in\mathcal{F}_{\mathcal{A}}\big\}$, $\mathcal{F}_{\mathcal{A}}=\cup _{\alpha \in \mathcal{A}}\mathcal{F}_{\alpha}$, with densities $p_f$ with respect to some common dominating measure.
Assume also that the models in $\mathcal{P}$ are such that for $d^2$ being $h^2$, $K$ or $V$, $d^2(p_f,p_{f_0})\lesssim \|f-f_0\|_2^2$.
Take an i.i.d.\ sample $\bm{X}^{(n)}=(X_1,\ldots, X_n)$, $X_i \sim p_{f_0}$, $f_0 \in\mathcal{F}_\alpha$, $\|f_0\|_\infty < M$, for some unknown smoothness $\alpha \in \mathcal{A}$.
Consider a prior $\Pi$ which verifies (\ref{A1-1}) through (\ref{A3}) for certain constants $c_1, c_2, c_3, t_1, t_2$ and $t_3$.
Assume also that either $\alpha\le 1$ or $t_2 \wedge t_3 =1$.

Then, for large enough $C>0$, $\Pi\big(f\in \mathcal{F}: h(p_f, p_{f_0})\ge C\epsilon_n | \bm{X}^{(n)}\big) \rightarrow 0$  as $n\to \infty$ in $P_0$-probability for $\epsilon_n = C_3 n^{-\alpha/(2\alpha+1)}(\log n)^{\alpha/(2\alpha+1) +(1- (t_1 \wedge t_3))/2}$.
If $h(p_f,p_{f_0}) \gtrsim \|f-f_0\|_2$ then in the previous statement the Hellinger distance may be replaced by the $L_2$ distance and the statement remains valid.
\end{thm}
\begin{proof}
We have that for some constant $k>0$ and $\mathcal{F}=\mathcal{S}$, $\mathcal{F}_n=\mathcal{S}_n$,
\begin{align*}
N \big(\epsilon_n,\mathcal{F}_n,h\big)	&	\le	N(\epsilon_n/k, \mathcal{F}_n,\|\cdot\|_2),\\
\Pi(\mathcal{F}\backslash\mathcal{F}_n)	&	=	P\big(s_{\bm{\theta}, \bm{K}_J}  \not\in \mathcal{F}_n\big),\\
\Pi(B(\bar{\epsilon}_n,f_0) )			&	\ge	P\big(\|s_{\bm{\theta}, \bm{K}_J} - f_0\|_\infty \le \bar{\epsilon}_n/k \big).
\end{align*}
The first inequality follows from the fact that by assumption $h(p_f,p_g) \le k \|f-g\|_2$ and so an $\epsilon/k$ cover of $ \mathcal{F}_n$ according to $\|\cdot\|_2$ induces an $\epsilon$ cover of $ \mathcal{F}_n$ according to $h$.
Then, since for $d^2$ being $K$ or $V$, $d^2(p_f,p_{f_0})\le k \|f-f_0\|_2^2$, we have $B(\bar{\epsilon}_n,f_0)\supset \big\{f\in\mathcal{F}: \|f-f_0\|_2 \le \epsilon/k \big\}$ and the last inequality follows.

By assumption $f_0\in\mathcal{F}_\alpha$ satisfies the conditions of Theorem \ref{theo}; assume (\ref{Ap}) holds for some $C_{f_0}$.
Consider then a prior that satisfies (\ref{A1-1})--(\ref{A3}).
Let us present a choice of quantities $M_n$, $\delta(j)$, $J_n$, $\bar{J}_n$, $\epsilon_n$ and $\bar{\epsilon}_n$  which meet conditions (\ref{C1})--(\ref{C4}).
First of all, sequence $M_n$ can be taken as a polynomial in $n$ (for instance, for normal or exponential conditional priors for $\bm{\theta}\in \mathbb{R}^j$ in (\ref{C2})) and $1/\delta(j)$ as a polynomial in $j$.
Next, note that there is no $\bar{J}_n$ that satisfies (\ref{C4}) unless $\alpha\le 1$  or $t_2 \wedge t_3 =1$. 
If either $\alpha>1$ or $t_2 \wedge t_3 < 1$, then the best possible choices are $\bar{J}_n = \tau C_{f_0}^{1/\alpha} (\bar{\epsilon}_n)^{-1/\alpha}$, $\bar{\epsilon}_n = C_1 (\log n/n)^{\alpha/(2\alpha+1)}$ for sufficiently large $C_1$, $J_n = C_2 n^{1/(2\alpha+1)} (\log n)^{2\alpha/(2\alpha+1) - t_1}$ for sufficiently large $C_2$, and finally, 
\[
\epsilon_n = C_3 n^{-\alpha/(2\alpha+1)}(\log n)^{\alpha/(2\alpha+1) +(1- t_1)/2}
\]
for sufficiently large $C_3$.
Since these quantities satisfy (\ref{C1})--(\ref{C4}), Theorem \ref{theo} implies  conditions (\ref{S1})--(\ref{S3})  for the quantities defined above.
Finally,  applying Theorem  \ref{theorem:contraction}, we conclude that the contraction rate of the resulting posterior is at most $\epsilon_n$, which appears to be optimal (up to a logarithmic factor) in a minimax sense over the H\"older class $\mathcal{H}_\alpha$ (also over $\alpha$-smooth Sobolev class).

\end{proof}

\begin{rmk}
A priori, it may be unknown whether $\alpha>1$ or not, or it may be simply known that $\alpha\le 1$. 
We can however always ensure the condition $t_2 \wedge t_3 < 1$ by an appropriate choice of prior. For example, we take a geometric prior on $J$ so that $t_2=0$ and a prior on $\bm{K}_j$ such that (\ref{A4}) (which implies (\ref{A2})) holds  with, say, $t_3=0$.
\end{rmk}

\begin{rmk}
The common practice, in applications, of endowing the location of the knots with a Poisson point process prior results in a prior that does not verify assumption (\ref{A0}).
Assumption (\ref{A0'}), however, permits this so long as a large enough point mass is placed at an equally spaced knot vector.
This very simple modification assures that our Theorem \ref{theo:main} may be applied to show that  these priors result in a rate adaptive posteriors.
\end{rmk}

\section[Examples]{Examples of Priors}
\label{examples}

We give now examples of particular choices for the several components of our hierarchical prior which verify conditions (\ref{A1-1}) through (\ref{A3}) and (\ref{A0'}).

As for the prior on the number of basis functions, assumptions (\ref{A1-1}) and (\ref{A1-2}) hold for the geometric, Poisson and negative binomial distributions; cf.\ \citet{Shen:2012}.
Assumption (\ref{A3}), on the other hand, will trivially hold if we assume, for example, the coordinates of $\bm{\theta}\in\mathbb{R}^j$ to be (conditionally on $J=j$) independent and identically distributed according to a density uniformly bounded away from zero on the interval $[-M,M]$.

There is an ample choice of priors on $\bm{K}_J$, given $J=j$, which satisfy condition (\ref{A0}).
First note that this condition enforces the prior on the location of the knots, for each $J=j$, to be such that, with probability 1, adjacent knots are at least $\delta(j)$ apart.
The function $1/\delta(j)$ can be taken as a polynomial in $j$ of high degree which makes the requirement less restrictive.
If a certain sequence $\epsilon_n$ verifies the conditions of Theorem \ref{theo} then an increase in the exponent of $1/\delta(j)$ can be accommodated by making $\epsilon_n$ larger by a multiplicative factor (cf.~condition (\ref{C1}).)

A simple choice for the prior on $\bm{K}_J$, given $J=j$, is to pick $(j-q)$ knots uniformly at random, without replacement, on a uniform $\delta(j)$-sparse grid.
This construction is possible if $\delta$ is chosen in such a way that $\lfloor 1/\delta(j)\rfloor>j-q$ for all $j$.
Another example is to take, for each $j$, the $(j-q)$ inner knots in $\bm{K}_j$ to be generated sequentially in the following way:
add a knot $K_1$ uniformly at random on the interval $[\delta(j),1-\delta(j)]$, then a knot $K_2$ uniformly at random on the interval $[\delta(j),1-\delta(j)]\backslash (K_1-\delta(j), K_1+\delta(j))$ 
and so on.
Finally, take the ordered $\bm{K}_j=(K_{(1)}, \dots, K_{(j-q)})$.
This construction is always possible if $1/\delta(j)$ grows faster than $2(j-q+1)$.
(If $J$ is Poisson distributed, these points are simply distributed like a homogeneous Poisson process, conditioned to have all points at least $\delta(J)$ apart.)
Note that for this construction, the probability $\mathbb{P}\big( m(\bm{K}_j) > \delta(j) | J = j \big)$ is at least $(1-2(j-q)\delta(j))^{j-q}$ which is very close to one if $j$ is large and $1/\delta(j)$ is a large power of $j$, say.
Clearly, condition (\ref{A0}) is satisfied for these two constructions since all prior mass is concentrated on partitions with sparseness larger than $\delta(j)$.

It is also easy to see that condition (\ref{A2}) is verified for the knot vectors obtained from one of these two constructions.
In fact, condition (\ref{A2}) is trivially fulfilled if, for some $0\le t_3 <1$,
\begin{equation}
\label{A4}
\mathbb{P}(\bm{K}_j= \bar{\bm{k}}_j)	\gtrsim	\exp\big(- c_3 j\log^{t_3}j\big), 
\end{equation}
where $\bar{\bm{k}}_j\in\mathcal{K}_j$ is the set of $(j-q)$ equally spaced inner knots.
This suggests a mechanism to assure that any prior which verifies (\ref{A0}) can be slightly modified to also verify (\ref{A2}):
given $J=j$, generate a Bernoulli random variable $X$ with success probability, say, $\exp(- c_3 j\log^{t_3}j)$;
if $X=1$, then take $\bm{K}_j = \bar{\bm{k}}_j$, otherwise pick the knots in $\bm{K}_j$ according to any procedure which verifies (\ref{A0}), for instance one of two procedures described above.
The resulting prior will trivially satisfy both (\ref{A0}) and (\ref{A2}).

Condition (\ref{A0}) necessarily excludes some partitions from the support of the prior (and then also from the support of the posterior.)
As mentioned before very few partitions will be excluded so long as $1/\delta(j)$ is a large enough power of $j$.
It is nonetheless of interest to design a weaker alternative for condition (\ref{A0}).
Condition $(\ref{A0'})$ plays this role, in that it allows priors on $\bm{K}$ which have \emph{any} partition of $[0,1]$ into non-empty intervals in its support.

Assuming condition (\ref{A0'}) instead of (\ref{A0}) consequently allows us to put positive mass on any vector of simple knots in a straightforward way:
generate a Bernoulli random variable with success probability $1-c_5\exp(-c_4 n)$;
if $X=1$ take $\bm{K}_j = \bar{\bm{k}}_j$, equally spaced;
if $X=0$ then take an arbitrary $\bm{K}_j$ (for example independent, uniformly distributed points on $[0,1]$.)
So long as we take $1/\delta(j)=j$ and $\tau \ge q$ then conditions (\ref{A0'}) and (\ref{A2}) are verified.
This procedure, although simpler, does place little prior mass on knot vectors with inhomogeneous distributions.

An alternative, less degenerate prior, which verifies (\ref{A0'}) and (\ref{A2}) can be obtained in the following way:
given $J=j$, first, generate a Bernoulli random variable $X_1$ with success probability $c_5\exp(-c_4 n)$;
if $X_1=1$ distribute the $(j-q)$ knots arbitrarily;
if $X_1=0$ then generate another Bernoulli random variable $X_2$ with success probability, $\exp(-j)$;
if $X_2=1$ then take $(j-q)$ equally spaced knots $\bar{\bm{k}}_j$;
If $X_2=0$, then place the knots such that (\ref{A0}) is verified.
This procedure should allow good control on the prior on the knots while not excluding any knot vectors.

Note that the priors described above which verify (\ref{A1-1}) through (\ref{A3}) do not depend on the sample size $n$, as prescribed by the Bayesian paradigm.
Condition (\ref{A0'}) is a weaker requirement then condition (\ref{A0}) but it will, introduce a dependence on the sample size $n$ in the prior.

\section[Technical results]{Technical results}
\label{aux}

In this section we collect some technical results.
Lemmas \ref{lemma0} and \ref{lemma1} are needed to bound the entropy number of the sieves $\mathcal{S}_n$ in Theorem \ref{theo}.
Lemma \ref{lemma2} claims in essence that if some bounds on the range of the function $f_0$ are known, then this knowledge can be incorporated into the prior on the coefficients $\bm{\theta}$.

Theorem 4.26 of \citet{Schumaker:2007} claims that if all the inner knots of a B-spline are simple, then the B-spline is continuous, uniformly over its support, with respect to its knots.
In Lemma \ref{lemma1} we establish a slightly stronger result (a Lipschitz-type property):
if we take two splines with the same coefficients in their respective B-spline basis, then the $L_\infty$ distance between the splines can be bounded by a multiple of the $l_\infty$ distance between the two sets of knots, as long as the sets of knots are sufficiently sparse.
First, we present a preliminary lemma.
Denote the $(r+1)$-th order divided difference of a function $h$ over the points $t_1, \dots, t_{r+1}$ as $[t_1, \dots, t_{r+1}]h = ([t_2, \dots, t_{r+1}]h - [t_1, \dots, t_r]h) / (t_{r+1}-t_1)$, with $[t_i]h=h(t_i)$. If $t_1 = \dots = t_{r+1}$ then $[t_1, \dots, t_{r+1}]h = h^{(r)}(t_1) / r!$ for a function $h$ with enough derivatives at $t_1$.

\begin{lem}
\label{lemma0}
Let $i\in\{1, \dots, r\}$, $r\ge 2$, $(k_1, \dots, k_{r+1}) \in (0,1)^{r+1}$.
Assume $k_{v+1}-k_v > \delta >0$ for $v=0,\dots,i-1,i+1,\dots,r$ and $k_{i+1} - k_i = 0$.
For fixed $x\in[0,1]$ take the function $h(y) = (x-y)_+^{q-1}$ with $y\in[0,1]$ and $q\ge 2$.
Then the divided difference $\big|[k_1, \dots, k_{r+1}]h\big| \le 4/\delta^r$ for $x\neq k_i$.
\end{lem}
\begin{proof}
Notice that $|h'(y)| = (q-1)(x-y)_+^{q-2} \le (q-1)\le 1/\delta$ for $x \neq y$, as $q \ge 2$ and thus $\delta  < k_2-k_1 < 1 \le \frac{1}{q-1}$.
Next, if $v=i-1$, $\big|[k_{v+1}, k_{v+2}]h\big| = |h'(k_{v+1})| \le 1/\delta$;
if $v\neq i-1$, $\big|[k_{v+1}, k_{v+2}]h\big| = |h(k_{v+2})-h(k_{v+1})|/|k_{v+2}-k_{v+1}| \le 2/\delta$.
We conclude $\big|[k_{v+1}, k_{v+2}]h\big|\le 2/\delta$ as long as $x\neq k_i$.

For $j=2, \dots, r$, define $\gamma_j = \min_{v=1, \dots, r+1-j} |k_{v+j}-k_v| \ge (j-1) \delta$.
Now we make use of Theorem 2.56 from \citet{Schumaker:2007} and the previous bound:
\[
\big|[k_1, \dots, k_{r+1}]h\big| \le \sum_{v=0}^{r-1} {r-1\choose v} 
\frac{\big|[k_{v+1}, k_{v+2}]h\big|}{\gamma_2 \dots \gamma_r} 
\le \frac{2^r}{(r-1)! \delta^r} \le \frac4{\delta^r}
\]
holds for all $x\neq k_i$. This completes the proof of the Lemma.

\end{proof}

\begin{lem}
\label{lemma1}
Let $\bm{\theta}\in\mathbb{R}^j$ satisfies $\|\bm{\theta}\|_\infty\le M$ and let $\bm{k},\bm{k}'\in \mathcal{K}^\delta_j = 
\{\bm{k}\in\mathcal{K}_j: m(\bm{k}) > \delta\}$ be such that $\|\bm{k}-\bm{k}'\|_\infty \le \delta$.
Then $\|s_{\bm{\theta},\bm{k}} - s_{\bm{\theta},\bm{k}'}\|_\infty \le L\|\bm{k}-\bm{k}'\|_\infty$, for $L=4 j (q+1) M\delta^{-(q+1)}$.
\end{lem}
\begin{proof}
Define $\bm{k}^l =(k^l_1,\ldots, k^l_{j-q})= (k'_1, \dots,k'_l,k_{l+1},\dots, k_{j-q})$ for $l=0,\dots, j-q$, such that $\bm{k}^0=\bm{k}$ and $\bm{k}^{j-q}=\bm{k}'$.
We get
\begin{align*}
\big\|s_{\bm{\theta},\bm{k}} & - s_{\bm{\theta},\bm{k}'}\big\|_\infty =\Big\| \sum_{i=1}^j \theta_i B_i^{\bm{k}^0} - \sum_{i=1}^j \theta_i B_i^{\bm{k}^{j-q}}\Big\|_\infty \le M\Big\|\sum_{i=1}^j (B_i^{\bm{k}^0} - B_i^{\bm{k}^{j-q}})\Big\|_\infty\\
&\le j M\max_{1\le i\le j}\big\|B_i^{\bm{k}^0} - B_i^{\bm{k}^{j-q}}\big\|_\infty\le j M\max_{1\le i\le j}\sum_{l=0}^{j-q-1} \big\|B_i^{\bm{k}^l} - B_i^{\bm{k}^{l+1}}\big\|_\infty\\
&\le	(q+1)  j M\max_{1\le i\le j} \max_{0\le l\le  j-q-1} \big\|B_i^{\bm{k}^l} - B_i^{\bm{k}^{l+1}}\big\|_\infty,
\end{align*}
The last inequality follows from (\ref{local_support}) and the fact that the inner knots of $B_i^{\bm{k}^l}$ and $B_i^{\bm{k}^{l+1}}$ differ only at the $(l+1)$-th entry. 

Theorem 4.27 of \citet{Schumaker:2007} gives explicit expressions for the derivative of a B-spline with respect to one of its knots.  These expressions are in terms of the divided differences which satisfy the conditions of Lemma \ref{lemma0}, so that combining this with Lemma \ref{lemma0} for $r=q+1$ (the maximal number of knots in the support of a B-spline) yields that this derivative is bounded in absolute value by $4\delta^{-(q+1)}$, except at $x=k_{l+1}^l$, where it is not defined.
Then, as $\|\bm{k}^l-\bm{k}^{l+1}\|_\infty \le \|\bm{k}-\bm{k}'\|_\infty$, we obtain that, for $x\not=  k_{l+1}^l$, $l=0,\ldots, j-q-1$,
\[
\big|B_i^{\bm{k}^l}(x) - B_i^{\bm{k}^{l+1}}(x)\big| \le |k_{l+1}^{l+1} -k_{l+1}^l |\sup_{k_{l+1}^l\in(0,1)}\Big| \frac{\partial B_i^{\bm{k}^l}(x)}{\partial k_{l+1}^l}\Big| \le\frac{4 \|\bm{k}-\bm{k}'\|_\infty}{\delta^{q+1}}.
\]
Since splines are continuous for all $q>1$, so is $s_{\bm{\theta},\bm{k}} - s_{\bm{\theta},\bm{k}'}$ and we conclude that the same bound must also hold for $x=k_{l+1}^l$.
Combining the above two relations concludes the proof.

\end{proof}

The properties of  B-splines allow to relate the range of the coefficients of the approximating spline to the range of the approximated function.
The following lemma generalizes Lemma 1 of \citet{Shen:2012} for non-equally spaced knots.

\begin{lem}
\label{lemma2}
Let $f\in \mathcal{F}^{\alpha}$ (so that (\ref{Ap}) holds), $a< b$,  $\varepsilon>0$.
Assume that  $f(x)\in [a+\varepsilon,b-\varepsilon]$ for all $x\in[0,1]$. 
Then there exits a positive constant $\delta=\delta(\mathcal{F}^{\alpha}, \varepsilon)$ such that for any $\bm{k}\in\mathcal{K}_j$, $j\ge q$, such that $M(\bm{k})\le\delta$, the coefficients $\bm{a}$ of the approximating spline $s_{\bm{a},\bm{k}}$ in (\ref{Ap}) can be taken to be contained in $(a,b)$.
\end{lem}
\begin{proof}
Fix $q$, $j$  and inner knots $\bm{k}$, assume $I = [a, b]$, $a<b$ and $a + \varepsilon < f < b - \varepsilon$, for some $\varepsilon>0$.

We use results from section 4.6 of \citet{Schumaker:2007} on dual basis of B-splines.
If $B_1^{\bm{k}}$, \dots, $B_j^{\bm{k}}$ is the B-spline basis associated with the inner knots $\bm{k}$, then there exists a dual basis $\lambda_1$, \dots, $\lambda_j$ of linear functionals such that, for each $i,r=1,\dots,j$, $\lambda_r B_i^{\bm{k}} = 1$ if $i=r$ and is $0$ otherwise.
As a consequence,  we obtain that $\lambda_i s_{\bm{a}, \bm{k}} = a_i$, and since  $\sum_{i=1}^j B_i^{\bm{k}}(x) =1$, it follows that  $\lambda_i c = c$ for any constant $c$ and all $i=1,\ldots,j$.
This dual basis is not necessarily unique and, according to Theorem 4.41 from \citet{Schumaker:2007}, can be taken such that $|\lambda_i f| \le C_1 \sup_{x\in I_i} | f(x)|$ where $I_i$ represents the support of $B_i^{\bm{k}}$ and constant $C_1$  depends only on $q$.
Each $I_i$ consists of at most $q$ adjacent intervals in the partition induced by $\bm{k}$ and thus the length of $I_i$ is bounded by $q M (\bm{k})$.

Let $s_{\bm{a}, \bm{k}}$ be such that (\ref{Ap}) is fulfilled for $f$.
Then for any constant $c$ 
\begin{eqnarray*}
| a_i - c | &=& \big|\lambda_i s_{\bm{a}, \bm{k}} - \lambda_i f + \lambda_i f - c \big| \le \big|\lambda_i (s_{\bm{a}, \bm{k}} -  f) | + |\lambda_i (f - c)\big|\\
&\le& C_1 C_f M^\alpha(\bm{k})+ C_1 \sup_{x\in I_i}| f(x) - c |.
\end{eqnarray*}
Take $c = \inf_{x\in I_i}f(x)$ and recall that $f\in \mathcal{F}_\alpha \subseteq
\mathcal{L}(\kappa_\alpha, L_\alpha)$.
Using the Lipschitz property, we derive that $\sup_{x\in I_i}| f(x) - c | = \sup_{x\in I_i} f(x) - \inf_{x\in I_i} f(x) \le L_\alpha (q\, M(\bm{k}))^{\kappa_\alpha}$ and therefore
\begin{eqnarray*}
| a_i - \inf_{x\in I_i}f(x)|	&\le& C_1 C_f M^\alpha(\bm{k}) + C_1 L_\alpha (q\, M(\bm{k}))^{\kappa_\alpha} \le C_2 M^{\alpha\wedge\kappa_\alpha}(\bm{k}).
\end{eqnarray*}
In the same way, if we take $c = \sup_{x\in I_i}f(x)$, we derive that $\sup_{x\in I_i}| f(x) - c|
\le  L_\alpha (q\, M(\bm{k}))^{\kappa_\alpha}$ and thus $\big|a_i - \sup_{x\in I_i}f(x)\big|\le C_2 M^{\alpha \wedge \kappa_\alpha}(\bm{k})$.

Now for $\delta = (\varepsilon/(2C_2))^{1/(\alpha \wedge \kappa_\alpha)}$ conclude that if $M(\bm{k})\le \delta$, then $a_i \ge \inf_{x\in I_i}f(x)  - C_2 M^{\alpha \wedge \kappa_\alpha}(\bm{k}) \ge\inf_{x\in I_i}f(x) - \varepsilon/2> a$.
For the same choice of $\delta$ we have $a_i\le\sup_{x\in I_i} f(x) + C_2 M^{\alpha \wedge \kappa_\alpha}(\bm{k})\le\sup_{x\in I_i}f(x) + \varepsilon/2< b$.

\end{proof}

\bibliographystyle{ba}
\bibliography{main}

\end{document}